\documentclass[12pt]{amsart}
\usepackage{amssymb, amsmath, amsthm, amsfonts, amscd}
\usepackage{xcolor}
\usepackage[english]{babel}
\usepackage{babel}
\usepackage{hyperref}
\usepackage{datetime}
\usepackage{mathtools}
\usepackage{mathabx}
\usepackage{enumerate}
\usepackage[mathscr]{euscript}
\usepackage{url}
\usepackage{enumitem}
\usepackage{mathtools}
\usepackage{commath}
\definecolor{cobalt}{rgb}{0.0, 0.28, 0.67}
\usepackage[capitalize,nameinlink]{cleveref}

\newtheorem{theorem}{Theorem}[section]
\newtheorem{lemma}[theorem]{Lemma}
\newtheorem{corollary}[theorem]{Corollary}
\newtheorem{proposition}[theorem]{Proposition}
\theoremstyle{definition}
\newtheorem{definition}[theorem]{Definition}

\theoremstyle{remark}
\newtheorem{remark}[theorem]{Remark}

\newcommand{\vp}{\varphi}
\newcommand{\clj}{\mathcal{J}}
\def \D{\mathbb{D}}

\def \T{\mathbb{T}}

\def \Z{\mathbb{Z}}
\newcommand{\clb}{\mathcal{B}}

\newcommand{\clm}{\mathcal{M}}
\newcommand{\cls}{\mathcal{S}}

\numberwithin{equation}{section}

\title[Commutant of complex symmetric weighted composition operators]{Operators commuting with complex symmetric weighted composition operators on $H^2$}

\author[Sudip Ranjan Bhuia]{Sudip Ranjan Bhuia}
\address{Indian Statistical Institute, Statistics and Mathematics Unit, 8th Mile, Mysore Road, Bangalore, 560059,
	India}
\email{sudipranjanb@gmail.com}

\date{\currenttime ;  \today}

\begin{document}
	
	\maketitle

\begin{abstract}
	In this paper, we initially study when an anti-linear Toeplitz operator is in the commutant of a composition operator.  Primarily, we investigate weighted composition operators $W_{g,\psi}$ commuting with complex symmetric weighted composition operators $W_{f,\varphi}$ on the Hardy space $H^2(\D)$. In particular, we give the descriptions of the symbols $g$ and $\psi$ such that the inducing weighted composition operator $W_{g,\psi}$ commutes with the complex symmetric weighted composition operator $W_{f,\varphi}$ with the conjugation $\clj$. Furthermore, we subsequently demonstrate that these weighted composition operators are normal and complex symmetric in accordance with the properties of the fixed point of the associated symbol $\vp$.
\end{abstract}
\section{Introduction and preliminaries}
Let $\mathcal{B}(H)$ be the algebra of all bounded linear operators on a separable complex Hilbert
space $H$. Given a fixed operator $T\in\mathcal{B}(H)$, we say an operator $S$ commutes with $T$ if $TS=ST$. The set of all operators which commute with $T$,
denoted $\{T\}^\prime$ that is, $\{T\}^\prime=\left\{S\in\mathcal{B}(H):ST-TS=0 \right\}$. It is well known that the set $\{T\}^\prime$ forms a weakly closed algebra which is called the commutant of $T$.

Let $\D$ denote the open unit disk in the complex plane $\mathbb{C}$. The Hardy space $H^2(\D)$ is the Hilbert space of the analytic functions $f$ on $\D$ having power series representations with square-summable complex coefficients. That is,
 \begin{equation*}
H^2(\D)=\left\{f: f(z)=\displaystyle\sum_{n=0}^{\infty}a_nz^n \quad\text{and}\quad \displaystyle\sum_{n=0}^{\infty}|a_n|^2<\infty \right\}.
 \end{equation*}
 
 The space $H^2(\D)$ is a Hilbert space with the inner product given by $$\langle f,g\rangle=\displaystyle\sum_{n=0}^{\infty}a_n\bar{b}_n,$$ where $f(z)=\displaystyle\sum_{n=0}^{\infty}a_nz^n$ and $g(z)=\displaystyle\sum_{n=0}^{\infty}b_nz^n$ are in  $H^2(\D)$.

Let $f$ be an analytic function on $\D$. Then $f$ is in $H^2(\D)$ if and only if
\begin{equation*}
\displaystyle\sup_{0<r<1}\frac{1}{2\pi}\int_{0}^{2\pi}|f(re^{i\theta})|^2d\theta<\infty.
\end{equation*} 
Moreover, the norm of such $f$ is given by
\begin{equation*}
\|f\|^2=\displaystyle\sup_{0<r<1}\frac{1}{2\pi}\int_{0}^{2\pi}|f(re^{i\theta})|^2d\theta<\infty.
\end{equation*} 
We define the Hilbert space $L^2(\mathbb{T})$ by the space of the square integrable functions
on $\mathbb{T}$, the unit circle in complex plane with respect to Lebesque measure, endowed with the inner product given by
\begin{equation*}
\langle f,g\rangle=\frac{1}{2\pi}\int_{0}^{2\pi}f(e^{i\theta})\overline{g(e^{i\theta})}d\theta,
\end{equation*} 
for all $f,g\in L^2(\mathbb{T})$.

The space $L^\infty(\mathbb{T})$ is the Banach space consisting of all essentially bounded measurable functions on $\mathbb{T}$. 

The Hardy space $H^2(\D)$ on the open unit disc $\D$ is identified with the closed subspace $H^2(\mathbb{T})$ of $L^2(\mathbb{T})$ consisting of functions $f$
on the boundary $\mathbb{T}$ whose negative Fourier coefficients vanish; the identification is given by the radial limit

\begin{equation*}
\tilde{f}(e^{i\theta}):=\lim_{r\rightarrow -1}f(re^{i\theta})\quad\text{for almost every } \theta\in [0,2\pi]
\end{equation*}
for $f\in H^2(\D)$.

Moreover, the reverse process is given by the Poisson integral formula
\begin{equation*}
f(re^{i\theta})=\frac{1}{2\pi}\int_{0}^{2\pi}\tilde{f}(e^{it})\frac{1-r^2}{1+r^2-2r\cos (\theta-t)},\quad re^{i\theta}\in \D.
\end{equation*}
Here $f$ is the harmonic extension of $\tilde{f}$ into the open unit disc $\D$ given by the Poisson integral formula. We
usually use the same symbol $f$ for $\tilde{f}$ and write $H^2(\D) = H^2(\mathbb{T})$ under the
identification. We denote by $H^\infty(\D)$ the space of all functions that are
analytic and bounded on $\D$. The space $H^\infty(\D)$ is a subspace of $H^2(\D)$.

 Let $\vp$ be a holomorphic self map of $\D$ and let $w \in \overline{\D}$. We say that $w$ is a \textit{fixed point} \cite[page 50]{Cowen:Book} of $\vp$ if
\[
\lim_{r \rightarrow 1^{-}} \vp(r w) = w.
\]
By a well known result \cite[page 51]{Cowen:Book}, if $w \in \T$ is a fixed point of $\vp$, then
\[
\vp'(w) = \lim_{r \rightarrow 1^{-}} \vp'(r w),
\]
exists as a positive real number or $+\infty$. Now let $\vp$ be an automorphism of $\D$. We say that $\vp$ is:
\begin{enumerate}
\item \textit{elliptic} if it has exactly one fixed point
situated in $\D$,
\item \textit{hyperbolic} if it has two distinct fixed points in $\T$, and
\item \textit{parabolic} if there is only one fixed point in $\T$.
\end{enumerate}
The Hardy space $H^2(\D)$ is a reproducing kernel Hilbert space with the kernel function $K_w$, where $K_w(z)=\frac{1}{1-\bar{w}z}$ for each $w\in \D$ with the property that $\langle f,K_w\rangle=f(w)$ for each $f\in H^2(\D)$. The linear span of the reproducing kernels $\{K_w:w\in \D\}$ is dense in $H^2(\D)$.

For a holomorphic self map $\vp$ on $\D$ and a holomorphic function $f$ on $\D$, the weighted composition operator $W_{f,\vp}$ on $H^2(\D)$ is defined by $W_{f,\vp}g =f\cdot (g\circ\vp)$ for all $g\in H^2(\D)$. The composition operator $C_\vp$ is defined by $C_\vp=W_{1,\vp}$. It is worth to note that $W_{f,\vp}=M_f C_\vp$ whenever $f\in H^\infty$, where $M_f$ denotes the multiplication operator. The action of the adjoint of weighted composition operator $W_{f,\vp}$ on the kernel function is given by \begin{equation*}
    W^*_{f,\vp}K_w=\overline{f(w)}K_{\vp(w)}.
\end{equation*} For more about composition operators, we refer the book \cite{Cowen:Book}.

 For $\varphi\in L^{\infty}$, the Toeplitz operator $T_{\varphi} : H^2 \rightarrow H^2$ is defined by the following formula
$$T_{\varphi} f = P (\varphi f )$$
for $f \in H^2$, where $P$ denotes the orthogonal projection of $L^2$ onto $H^2$ . It is known that $T_{\varphi}$ is bounded if and only if $\varphi \in L^{\infty}$ and $\left\|T_{\varphi}\right\| = \left\|\varphi\right\|_{\infty}$. A Toeplitz operator $T_{\varphi}$ is called analytic if $\varphi\in H^{\infty}$, that is,
$\varphi$ is a bounded analytic function on the unit disc $\mathbb{D}$.
\begin{definition}\label{d1.4}
	For an bounded anti-linear operator $X$ on $H$, there is a unique anti-linear operator $X^{\#}$ called the anti-linear adjoint of $X$, if it satisfies the following relation
	\begin{equation}
		\langle Xx,y \rangle =\overline{\langle x, X^{\#}y \rangle}
	\end{equation}
	for all $x,y \in H$.
\end{definition}
The anti-linear operator $X$ is called anti-linear self-adjoint if $X^{\#}=X$. Denote $\clb_{a}(H)$ by the collection of all anti-linear bounded operators on $H$.
\begin{definition}
A conjugation on a separable complex Hilbert space $H$ is an anti-linear operator $C$ on
$H$ which satisfies the following conditions
\begin{enumerate}
	\item $C$ is isometric: $\langle Cx, Cy \rangle = \langle y, x \rangle$, $x, y \in H $,
	\item $C$ is involutive: $C^2  = I$.
\end{enumerate}
 where $I$ is the identity operator on $H$.

\end{definition}
\begin{definition}
		An anti-linear operator $C$ on a separable complex Hilbert space $H$ is a conjugation if and only if it is both unitary and self-adjoint.
\end{definition}
We say that $T$ is $C$-symmetric if $T = CT^* C$, and complex symmetric if there exists a conjugation $C$ with
respect to which $T$ is $C$-symmetric.

In the Hardy space $H^2(\D)$, the conjugation operator $\mathcal{J}:H^2(\D)\rightarrow H^2(\D)$ is defined by \[\mathcal{J}f(z)=\overline{f(\bar{z})},\] for all $z\in \D$ and $f\in H^2(\D)$. In fact every conjugation $C$ on the Hardy space $H^2(\D)$ is unitarily equivalent to the conjugation operator $\clj$ (see \cite[p.~172]{RG2}).

 For more details on complex symmetric operators readers are referred to \cite{Garcia:Prodan:Putinar, Garcia:Putinar:I, Garcia:Putinar:II, Garcia:Wogen}. The study of complex symmetric weighted composition operators on the Hardy space was initiated by Garcia and Hammond in \cite{Garcia:Hammond}. In \cite{Jung}, authors gave a classification of complex symmetric weighted composition operators with respect to a special conjugation. Generally, providing information about the operators that commute with a specific operator offers insights into the operator's structure. The commutant of a particular operator is a relatively rare subject of study. C. C. Cowen's research has focused on examining the commutants of composition and specific Toeplitz operators, as indicated in references \cite{Cowen:AnaToeplitz} and \cite{Cowen:CommAnaFunc}. Additionally, B. Cload has contributed findings concerning the commutants of composition operators in \cite{Cload:PhDThesis}. In \cite{Bhuia:Pradhan:Sarkar}, it is observed that any anti-linear operator $A\in \clb_{a}(H^2)$ satisfies $M^*_zAM_z=A$ if and only if $A=T_\vp \clj$ which is called anti-linear Toeplitz operator. Following B. Cload, we are interested to study properties of anti-linear Toeplitz operator whenever it commutes with a composition operator.

 In \cite{Cowen:EKo}, authors have studied the self-adjoint weighted composition operators on the Hardy space of unit disc. E. Ko studied the commutant of self-adjoint weighted composition operators in \cite{EKo}. It would be interesting to classify weighted composition operators commuting with complex symmetric weighted composition operators $W_{f,\varphi}$ with certain conjugation  on $H^2$. In Section $2$, we address the following question: ``Let $A\in \clb_{a}(H^2)$ be such that $M^*_zAM_z=A$ and $C_\vp A=A C_\vp$. Then what can we say about $A$?".  Theorem $2.1$ answers this question, showing that such anti-linear operators are given by $A=T_f\clj$, where $f$ is an analytic function.

  In section $3$, we give description of the weighted composition operators that commute with complex symmetric weighted composition operator with the conjugation $\clj$.

\begin{theorem}\cite[Theorem 6]{Cowen:Guna:Ko}\label{SA_WCO_WHrady}
	Let $\gamma\in \mathbb{N}$. Let $f$ be in $H^\infty$ and let $\vp$ be an analytic map of the unit disc into itself. If the weighted composition operator $W_{f,\vp}$ is Hermitian on $H_\gamma(\D)$ , then
	$f(0)$ and $\vp^\prime(0)$ are real and
	\begin{equation}
	\vp(z)=a_0+\frac{a_1z}{1-\bar{a}_0z}\quad\text{and}\quad f(z)=\frac{c}{(1-\bar{a}_0z)^\gamma}
	\end{equation} 
	where  $a_1=\vp^\prime(0)$, and $c=f(0)$.

	Conversely, let $a_0\in \D$, and let $c$ and $a_1$ be real numbers. If $\vp(z)=a_0+\frac{a_1z}{1-\bar{a}_0z}$ maps the unit disc into itself and $f(z)=\frac{c}{(1-\bar{a}_0z)^\gamma}$, then the weighted
	composition operator $W_{f,\vp}$ is Hermitian.
\end{theorem}

\begin{theorem}\cite[Theorem 3.5]{EKo}\label{Comtnt_SA_WCO}
	Let $g\in H^\infty$ and let $\psi$ be an analytic map of $\D$ into itself. Assume that $W_{f,\vp}$ is self-adjoint on $H^2(\D)$ and $\vp$, not an elliptic automorphism, has a fix point $b\in \D$. Then $W_{g,\psi}\in \{W_{f,\vp}  \}^\prime$ if and only if $W_{g,\psi}$ has the following symbols; for $b\neq 0$,
	\begin{equation*}
	\psi(z)=d_0+\frac{d_2z}{1-d_1z}\quad\text{and}\quad g(z)=g(b)\frac{d_3}{1-d_1z},
	\end{equation*}
	where $\psi(0)=d_0=\frac{(\alpha-1)b}{|b|^2\alpha-1},\, d_1=\frac{(\alpha-1)\bar{b}}{|b|^2\alpha-1},\, \psi^\prime(0)=d_2=\alpha\frac{(|b|^2-1)^2}{(|b|^2\alpha-1)^2},\, d_3=\frac{|b|^2-1}{|b|^2\alpha-1}$ for some $\alpha\in \mathbb{C}$ and for $b=0$,
	\begin{equation*}
	\psi(z)=\alpha z\quad\text{and}\quad g(z)=g(0)
	\end{equation*}
	for some $\alpha\in \mathbb{C}$.
\end{theorem}
\begin{theorem}\cite[Theorem 3.3]{Jung}\label{CSWCO_charac}
	Let $\varphi$ be an analytic self-map of $\D$ and $f\in H^\infty(\D)$ be not identically zero. If the weighted composition operator $W_{f,\varphi}$ is complex symmetric defined on $H^2(\D)$ with the conjugation $\mathcal{J}$, then 
	\begin{equation*}
	f(z)=\frac{b}{1-a_0 z}\quad\varphi(z)=a_0+\frac{a_1 z}{1-a_0 z},
	\end{equation*} 
	where $a_0=\varphi(0),\,a_1=\varphi^\prime(0)$, and $b=f(0)$.
	
	Conversely, let $a_0\in \D$. If $\varphi(z)=a_0+\frac{a_1 z}{1-a_0 z}$ maps the unit disc into itself and $f(z)=\frac{b}{1-a_0 z}$, then the weighted composition operator $W_{f,\varphi}$ is complex symmetric with the conjugation $\mathcal{J}$.
\end{theorem}
\begin{lemma}\cite{Martin:Vukotic}\label{LFT}
	A linear fractional map $\phi$ of the form $\phi(z)=\frac{az+b}{cz+d}$;  $ad-bc\neq0$, maps $\D$ into itself if and only if
	\begin{equation} 
	|b\bar{d}-a\bar{c}|+|ad-bc|\leq|d|^2-|c|^2.
	\end{equation}
\end{lemma}

\begin{lemma}\cite[Lemma 4.8]{Jung}\label{DtoD-condition}
	Let $\varphi(z)=a_0+\frac{a_1 z}{1-a_0 z}$. Then $\varphi$ maps the open unit disc into itself if and only if $|a_0|<1$ and $2|a_0+\bar{a}_0(a_1-a_0^2)|\leq 1-|a_1-a_0^2|^2$.

	In particular, when $a_1 = a_0^2 = 0$, $\varphi$ maps the open unit disc into itself if and only if
	$|a_0 | \leq \frac{1}{2}$, and when $a_1 -a_0^2 = \pm 1$, $\varphi$ maps the open unit disc into itself if and only if
	$a_0$ is either real or purely imaginary.
\end{lemma}
\begin{proposition}\cite[Proposition 4.4]{Jung}\label{Eigenvalue_CS}
	Let $\varphi$ be an analytic self-map of $\D$ and let $f\in H^\infty(\D)$ be not identically zero on $\D$, where $\varphi(0)\neq 0,\,\varphi^\prime(0)\neq 0$, and $\varphi(\lambda)=\lambda$ for some $\lambda\in \D$. If $W_{f,\varphi}$ is complex symmetric with the conjugation $\mathcal{J}$, then
	\begin{equation*}
	g_j(z):=\frac{1}{1-\lambda z}\left(\frac{\lambda-z}{1-\lambda z} \right)^j
	\end{equation*}
	 is an eigenvector of $W_{f,\varphi}$ with respect to the eigenvalue $f(\lambda)\varphi^\prime(\lambda)^j $ for each non-negative integer $j$. 
\end{proposition}

\section{Anti-linear Toeplitz operators in the commutant of a composition operator}
B. Cload (cf. \cite[Theorem 2]{Cload:PhDThesis})  proved that if $\vp:\D\rightarrow\D$ be an analytic mapping which is neither an elliptic disc automorphism of finite periodicity nor the identity mapping and $f\in L^\infty$ such that $T_fC_\vp=C_\vp T_f$, then $f$ is analytic, that is, $T_f$ is an analytic Toeplitz operator.

 In this section, we aim to investigate operators $A\in \clb_{a}(H^2)$ that satisfy the following equations: $M^*_zAM_z=A$ and $AC_\vp=C_\vp A$. The proof follows a similar approach to related problems, but there is a notable difference due to the anti-linearity of $A$. This anti-linearity necessitates the introduction of an additional condition in the hypothesis to accommodate the unique properties of these operators.

\begin{theorem}
Let $\vp:\D\rightarrow\D$ be an analytic mapping which is neither an elliptic disc automorphism of finite periodicity nor the identity mapping. Let $f\in L^\infty$ such that $f(z)=\displaystyle\sum_{n=-\infty}^{\infty} c_nz^n$ with $c_0=0$ and $T_f\clj C_\vp=C_\vp T_f\clj$, then $f$ is analytic.
\end{theorem}

\begin{proof}
	The proof breaks into three parts depending on the nature of $\vp$.
\begin{enumerate}
	\item[{\bf Case 1:}] Suppose that $\vp$ is neither elliptic disc automorphism nor a constant. Let $a$ be the Denjoy-Wolff point of $\vp$. Denote $A=T_f\clj$. Let $f=f_1+f_2$, where $f_1\in (H^2)^\perp$ and $f_2\in H^2$. Then for all $l\geq 0$,
	
	\begin{equation*}
		\begin{split}
			A(z^l)=T_f\clj z^l=T_f z^l=P(z^l f)=c_{-l}+c_{-l+1}z+\cdots+c_{-1}z^{l-1}+z^l f_2
		\end{split}
	\end{equation*} Now, $AC_\vp=C_\vp A$ implies $C_{\vp^n}A=AC_{\vp^n}$.

	Thus for all $l\geq 0$, we have
	\begin{equation*}
		\begin{split}
			C_{\vp^n}Az^l&=AC_{\vp^n}z^l\\
			c_{-l}+c_{-l+1}\vp^n+\cdots+c_{-1}(\vp^n)^{l-1}+(\vp^n)^l f_2\circ \vp^n&=A(\vp^n)^l
		\end{split}
	\end{equation*}	
	
	Also, note that $AC_\vp 1=C_\vp A1$ which implies $f_2=f_2\circ \vp$. Thus we have 
	\begin{equation*}
		A(\vp^n)^l=	c_{-l}+c_{-l+1}\vp^n+\cdots+c_{-1}(\vp^n)^{l-1}+(\vp^n)^l f_2.
	\end{equation*}

	By taking inner product on both sides with $1$ and using the fact that for each fixed $l\geq 0$, the sequence $(\vp^n)^l$ coverges weakly to $a^l$ (cf. \cite[Lemma 2 ]{Cload:PhDThesis}), we obtain
	
	\begin{equation*}
	\bar{a}^lf_2(0)=c_{-l}+c_{-l+1}a+\cdots+c_{-1}a^{l-1}+a^l f_2(0).
	\end{equation*}
Now $l=1$, we get $\bar{a}f_2(0)=c_{-1}+af_2(0)$ and this implies $c_{-1}=f_2(0)(\bar{a}-a)$.

for $l=2$, we have $\bar{a}^2f_2(0)=c_{-2}+c_{-1}a+a^2f_2(0)$ then by substituting the value of $c_{-1}$, we get $c_{-2}=f_2(0)(\bar{a}^2-|a|^2)=f_2(0)\bar{a}(\bar{a}-a)$. Similarly, we can show that $c_{-l}=f_2(0)(\bar{a})^{l-1}(\bar{a}-a)$ for any $l\geq 1$. Since $f_2(0)=c_0=0$, we get $c_{-l}=0$ for all $l\geq 1$, and this implies $f$ is analytic.

\item[{\bf Case 2:}] Suppose $\vp$ is a nonzero constant that is, $\vp(z)=b$ for all $z\in \D$.	Then $AC_\vp 1=C_\vp A 1$ will imply $f_2=f_2(b)$, that is, $f_2$ is constant function in $H^2$. Since $f_2(0)=0$, $f_2$ is an identically zero function.
	
For any $g\in H^2$, 
\begin{equation*}
	\begin{split}
	C_\vp A g&=AC_\vp g\\
	C_\vp P(f\tilde{g})&=A g(b),\quad\text{where $\tilde{g}(z)=\overline{g(\bar{z})}$}\\
	(P(f\tilde{g}))(b)&= \overline{g(b)}f_2=0\\
	(P(f_1\tilde{g}))(b)&=0\\
	\end{split}
\end{equation*}	
	Note that if $g(z)=z^k$ then $\tilde{g}(z)=z^k$. Therefore, take $g$ to be $z^k$ successively, where $k\geq 1$ to conclude that $c_{-k}=0$ for all $k\geq 1$. This will imply $f=0$, an analytic function.

\item[{\bf Case 3:}] Let $\vp$ be an elliptic automorphism of infinite peridicity. First we assume that the fix point be $0$. Then by Schwarz's lemma $\vp(z)=e^{i\theta}z$, where $e^{in\theta}\neq 1$ for all $n\in \Z$. Let $A\in \clb_{a}(H^2)$ such that $AC_\vp=C_\vp A$. Let $\{z^n:n\geq 0\}$ be the standard orthonormal basis for $H^2$. Then we see that
\begin{equation*}
	\langle A C_\vp z^k,z^l\rangle=\langle Ae^{ik\theta} z^k,z^l\rangle=e^{-ik\theta}\langle A z^k,z^l\rangle=e^{-ik\theta} a_{lk}
\end{equation*}

\begin{equation*}
	\langle C_\vp Az^k,z^l\rangle=\left\langle C_\vp\sum_{m=0}^{\infty}a_{mk}z^m,z^l\right\rangle=\left\langle \sum_{m=0}^{\infty}a_{mk}e^{im\theta}z^m,z^l\right\rangle=a_{lk}e^{il\theta}
\end{equation*}

This implies $(e^{i(l+k)\theta}-1 )a_{lk}=0$. Which shows that $a_{lk}=0$ for all $l,\, k$. Thus, $A$ is an operator whose matrix representation has all entries equal to zero with respect to the orthonormal basis $\{z^n:n\geq 0\}$.

\begin{equation*}
0=a_{lk}=	\langle Az^k,z^l\rangle=\langle T_f\clj z^k,z^l\rangle=\langle T_f z^k,z^l\rangle
\end{equation*}
implies $f=0$.

Now suppose that the fix point be $b$ which is nonzero. Let $\alpha(z)=\frac{b-z}{1-\bar{b}z}$ for all $z\in \D$. Again by writing $f=f_1+f_2$ ans using the relation $AC_\vp 1=C_\vp A1$ implies $f_2=f_2\circ\vp$. Thus $f_2$ is constant as $\vp$ has infinite periodicity. In fact $f_2=0$ as $f_2(0)=0$. Thus $T_{f_1}\clj C_\vp=C_\vp T_{f_1}\clj$. Denote $A_1=T_{f_1}\clj $. Then $A_1 C_\vp=C_\vp A_1$ and thus $C_\alpha A_1 C_\alpha C_\alpha C_\vp C_\alpha=C_\alpha C_\vp C_\alpha C_\alpha A_1 C_\alpha $. Therefore, $C_\alpha A_1 C_\alpha $ commutes with $C_{\alpha\circ\vp\circ\alpha}$. Since $\alpha\circ\vp\circ\alpha$ is elliptic disc automorphism of infinite periodicity with fix point $0$, the previous argument implies that $f_1=0$. Thus, we conclude that $f=0$, an analytic map.
\end{enumerate}

\end{proof}

The following result demonstrates that the previous theorem cannot be extended to all elliptic disc automorphisms. The proof is similar to the one found in  \cite[Theorem 3]{Cload:PhDThesis}, and therefore, it is omitted here.
\begin{theorem}
	Let $\vp$ be an elliptic disc automorphism of period $q\, (\geq 2)$ with $\vp(0)=0$ and $f(z)=\displaystyle\sum_{n=-\infty}^{\infty}a_{n}z^{n}\in L^\infty$. Then $T_f\clj$ commutes with $C_\vp$ if and only if $f(z)=\displaystyle\sum_{n=-\infty}^{\infty}a_{nq}z^{nq}$.
\end{theorem}

\section{Description of commutant}
Throughout this section, we consider $W_{f,\varphi}$ is complex symmetric with the conjugation $\mathcal{J}$ and $\varphi$ is not an identity map. Because if $\varphi$ is an identity map then $W_{g,\psi}\in \{W_{f,\varphi}\}^\prime$ will always holds. That is, in view of Theorem \ref{CSWCO_charac}, we will always assume that $a_1\neq 1$.
\begin{lemma}\label{common fix point}
	Let $g\in H^\infty$ and $\psi$ be an analytic map of $\D$ into itself. Assume that $W_{g,\psi}\in \{W_{f,\varphi} \}^\prime$, where $W_{f,\varphi}$ is complex symmetric with the conjugation $\mathcal{J}$. If $\varphi$ has a fixed point in $\D$, then $\psi$ has the same fixed point with $\varphi$.
\end{lemma}
\begin{proof}
	Since $W_{f,\varphi}$ is $\mathcal{J}$-symmetric with the conjugation $\mathcal{J}f(z)=\overline{f(\bar{z})}$, it follows from Theorem \ref{CSWCO_charac} that 
	\begin{equation}\label{symbols_CS_structure}
	f(z)=\frac{b}{1-a_0 z},\quad\varphi(z)=a_0+\frac{a_1 z}{1-a_0 z},
	\end{equation} 
	where $a_0=\varphi(0),\,a_1=\varphi^\prime(0)$, and $b=f(0)$.
	
	Let $\lambda$ be a fixed point of $\varphi$ in $\D$. It is very clear that if $a_0=0$, then $\lambda=0$. When $a_0\neq 0$, then the fixed point $\lambda$ is given by
	\begin{equation}\label{root lambda}
	\lambda=\frac{2a_0}{1+a_0^2-a_1\mp\sqrt{(1+a_0^2-a_1)^2-4a_0^2}}.
	\end{equation}
	{\bf Case 1:} If $\lambda=0$, then $\varphi(z)=a_1z$. Since $W_{g,\psi}\in \{W_{f,\varphi}\}^\prime$, we have
	$$\psi(a_1z)=\psi(\varphi(z))=\varphi(\psi(z))=a_1\psi(z).$$
	Now, consider the power series $\psi(z)=\displaystyle\sum_{n=0}^{\infty}\beta_nz^n$, then we have 
	\begin{equation*}
	\displaystyle\sum_{n=0}^{\infty}\beta_na_1^nz^n=\sum_{n=0}^{\infty}a_1\beta_nz^n.
	\end{equation*}
	Hence $\beta_0=a_1\beta_0$. Since $a_1\neq 1$, we have $\beta_0=0$ and this implies that $\psi(0)=0$.
	
{\bf Case 2:} If $\lambda\neq 0$, then consider the kernel function $K_\lambda$, that is, $K_\lambda(z)=\frac{1}{1-\bar{\lambda}z}$ for all $z\in \D$. We know that $W_{f,\varphi}^*K_\lambda=\overline{f(\lambda)}K_{\varphi(\lambda)}$ and since $W_{g,\psi}\in \{W_{f,\varphi}\}^\prime$, we have
\begin{equation*}
W^*_{f,\varphi}W^*_{g,\psi}K_\lambda=W^*_{g,\psi}W^*_{f,\varphi}K_\lambda=\overline{f(\lambda)}W^*_{g,\psi}K_{\varphi(\lambda)}=\overline{f(\lambda)}W^*_{g,\psi}K_\lambda.
\end{equation*}

{\bf Case 2(i):} If $W^*_{g,\psi}K_\lambda=0$, then $K_\lambda$ is the eigenvector of $W^*_{g,\psi}$ corresponding to the eigenvalue $0$. Since $\varphi(\psi(z))=\psi(\varphi(z))$, we get $\varphi(\psi(\lambda))=\psi(\varphi(\lambda))=\psi(\lambda)$, and this implies that $\psi(\lambda)$ is a fixed point of $\varphi$, and hence $\psi(\lambda)=0$, or $\lambda$. Now if $\psi(\lambda)=0$, then $a_0=\varphi(0)=\varphi(\psi(\lambda))=\psi(\varphi(\lambda))=\psi(\lambda)=0$, a contradiction to the fact that $a_0\neq 0$. Therefore, $\psi(\lambda)=\lambda$.
	
{\bf Case 2(ii):} If $W^*_{g,\psi}K_\lambda\neq 0$, then $W^*_{g,\psi}K_\lambda$ is an eigenvector of $W^*_{f,\varphi}$ with an eigenvalue $\overline{f(\lambda)}$. Since $W_{f,\varphi}$ is $\mathcal{J}$-symmetric, we have $\mathcal{J}W^*_{g,\psi}K_\lambda$ is an eigenvector corresponding to the eigenvalue $f(\lambda)$. Therefore, by Proposition \ref{Eigenvalue_CS}, we get $\mathcal{J}W^*_{g,\psi}K_\lambda=\beta K_{\bar{\lambda}}$ for some nonzero complex number $\beta$, and this implies $g(\lambda)K_{\overline{\psi(\lambda)}}=\beta K_{\bar{\lambda}}$. Therefore, $\psi(\lambda)=\lambda$. Thus $\lambda$ is a fixed point of $\psi$. This completes the proof.
	
\end{proof}
\begin{lemma}\label{relation among the coefficients}
Let $d_0=\frac{\lambda (\alpha-1)}{\lambda^2\alpha-1}=d_1$, $d_2=\alpha\frac{(\lambda^2-1)^2}{(\lambda^2\alpha-1)^2}$, where $\alpha$ is any complex number and $\lambda$ as in Equation \ref{root lambda}. Then the following are true
\begin{enumerate}
	\item\label{relation1} $d_1+a_0(d_2-d_0d_1)=a_0+d_1(a_1-a_0^2)$.
	\item\label{relation2} $\bar{d}_0(d_2-d_0^2-1)=-\bar{\lambda}(\lambda^2+1)\frac{|1-\alpha|^2}{|\lambda^2\alpha-1|^2}$.
\end{enumerate}
\end{lemma}
\begin{proof}
{\bf Proof of (\ref{relation1}):}  To prove the relation in (\ref{relation1}), it is enough to prove $a_0(d_2-d_0d_1-1)-d_1(a_1-a_0^2-1)=0$.
Now
\begin{equation*}
\begin{split}
a_0(d_2-d_0d_1-1)&=a_0(d_2-d_0^2-1)\\
&=a_0\left(\frac{\lambda^2-\alpha}{\lambda^2\alpha-1}-1\right)\\
&=a_0\frac{(1-\alpha)(\lambda^2+1)}{\lambda^2\alpha-1}
\end{split}
\end{equation*}
and
\begin{equation*}
\begin{split}
d_1(a_1-a_0^2-1)=\frac{\lambda(\alpha-1)(a_1-a_0^2-1)}{\lambda^2\alpha-1}.
\end{split}
\end{equation*}
Therefore,
\begin{equation*}
\begin{split}
&a_0(d_2-d_0d_1-1)-d_1(a_1-a_0^2-1)\\&=a_0\frac{(1-\alpha)(\lambda^2+1)}{\lambda^2\alpha-1}-\frac{\lambda(\alpha-1)(a_1-a_0^2-1)}{\lambda^2\alpha-1}\\
&=\frac{1-\alpha}{\lambda^2\alpha-1}\left[a_0(\lambda^2+1) +\lambda(a_1-a_0^2-1)\right]\\
&=\frac{1-\alpha}{\lambda^2\alpha-1}\left[a_0(\lambda^2+1) -\lambda(1+a_0^2-a_1)\right]\\
&=\frac{1-\alpha}{\lambda^2\alpha-1}\left[ a_0\Bigg\{\left(\frac{1+a_0^2-a_1\pm\sqrt{(1+a_0^2-a_1)^2-4a_0^2}}{2a_0} \right)^2+1\Bigg\}\right.\\&\hspace{2cm}\left.-\frac{1+a_0^2-a_1\pm\sqrt{(1+a_0^2-a_1)^2-4a_0^2}}{2a_0}(1+a_0^2-a_1)\right]\\
&=\frac{1-\alpha}{\lambda^2\alpha-1}\left[\frac{(1+a_0^2-a_1)^2\pm(1+a_0^2-a_1)\sqrt{(1+a_0^2-a_1)^2-4a_0^2}}{2a_0} \right.\\&\hspace{2cm}\left.-\frac{1+a_0^2-a_1\pm\sqrt{(1+a_0^2-a_1)^2-4a_0^2}}{2a_0}(1+a_0^2-a_1)\right]\\
&=0.
\end{split}
\end{equation*}
This completes the proof of (\ref{relation1}).

{\bf Proof of (\ref{relation2}):} We have 	
\begin{equation}
\begin{split}
d_2-d_0^2-1&=\alpha\frac{(\lambda^2-1)^2}{(\lambda^2\alpha-1)^2}-\frac{\lambda^2 (\alpha-1)^2}{(\lambda^2\alpha-1)^2}-1\\
&=\frac{\alpha\lambda^4-2\lambda^2\alpha+\alpha-\lambda^2\alpha^2+2\lambda^2\alpha-\lambda^2}{(\lambda^2\alpha-1)^2}-1\\
&=\frac{\alpha\lambda^4+\alpha-\lambda^2\alpha^2-\lambda^2}{(\lambda^2\alpha-1)^2}-1\\
&=\frac{\lambda^2-\alpha}{\lambda^2\alpha-1}-1\\
&=\frac{\lambda^2-\alpha-\lambda^2\alpha+1}{\lambda^2\alpha-1}\\
&=\frac{(1-\alpha)(\lambda^2+1)}{\lambda^2\alpha-1}.
\end{split}
\end{equation}
Therefore,
\begin{equation*}
\bar{d}_0(d_2-d_0^2-1)=\frac{\bar{\lambda}(\bar{\alpha}-1)}{\overline{\lambda^2\alpha-1}}\times\frac{(1-\alpha)(\lambda^2+1)}{\lambda^2\alpha-1}=-\bar{\lambda}(\lambda^2+1)\frac{|1-\alpha|^2}{|\lambda^2\alpha-1|^2}
\end{equation*}
\end{proof}
\begin{remark}\label{lambda real}
	$\bar{d}_0(d_2-d_0^2-1)=d_0(\overline{d_2-d_0^2}-1)$ if and only if $\lambda$ is real.
\end{remark}
\begin{theorem}
	Let $g\in H^\infty$ and let $\psi$ be an analytic map of $\D$ into itself. Assume that $W_{f,\varphi}$
	is complex symmetric with the conjugation $\mathcal{J}$ on $H^2(\D)$ and $\varphi$, not an elliptic automorphim, has a fixed point $\lambda$ in $\D$ . Then
	$W_{g,\psi}\in \{W_{f,\varphi}\}^\prime$ if and only if $W_{g,\psi}$ has the following symbol functions; for $\lambda\neq 0$,
	\begin{equation*}
	\psi(z)=\frac{(\lambda^2-\alpha)z+\lambda (\alpha-1)}{\lambda (1-\alpha)z+(\lambda^2\alpha-1)}\quad\text{and}\quad g(z)=g(\lambda)\frac{\lambda^2-1}{\lambda (1-\alpha)z+(\lambda^2\alpha-1)}
	\end{equation*}
	and for $\lambda=0$, 
	\begin{equation*}
	\psi(z)=\alpha z\quad\text{and}\quad g(z)=g(0)
	\end{equation*}
	for some $\alpha\in \mathbb{C}$.
\end{theorem}
\begin{proof}
	Let $\lambda\neq 0$. Since 	$W_{g,\psi}\in \{W_{f,\varphi}\}^\prime$, we have 
	\begin{equation*}
	\begin{split}
	W_{f,\varphi}W_{g,\psi}K_{\bar{\lambda}}&=W_{g,\psi}W_{f,\varphi}K_{\bar{\lambda}}\\
	&=W_{g,\psi}\mathcal{J}W^*_{f,\varphi}\mathcal{J}K_{\bar{\lambda}}\\
	&=W_{g,\psi}\mathcal{J}W^*_{f,\varphi}K_{\lambda}\\
	&=f(\lambda)W_{g,\psi}K_{\bar{\lambda}}.
	\end{split}
	\end{equation*}
	
If $W_{g,\psi}K_{\bar{\lambda}}=0$, then $K_{\bar{\lambda}}$ is an eigenvector of $W_{g,\psi}$ corresponding to an eigenvalue $0$.

If $W_{g,\psi}K_{\bar{\lambda}}\neq0$, then If $W_{g,\psi}K_{\bar{\lambda}}$ is an eigenvector of $W_{f,\varphi}$ corresponding to the eigenvalue $f(\lambda)$. Then by Proposition \ref{Eigenvalue_CS}, we have $W_{g,\psi}K_{\bar{\lambda}}=\gamma_1 K_{\bar{\lambda}}$ for some nonzero $\gamma_1\in \mathbb{C}$. Again by Proposition \ref{Eigenvalue_CS}, the function 
$\eta=\frac{1}{1-\lambda z}\left(\frac{\lambda-z}{1-\lambda z} \right)$ is an eigenvector of $W_{f,\varphi}$ corresponding to the eigenvalue $f(\lambda)\varphi^\prime(\lambda)$. Since 	$W_{g,\psi}\in \{W_{f,\varphi}\}^\prime$, we have
\begin{equation*}
W_{f,\varphi}W_{g,\psi}\eta=W_{g,\psi}W_{f,\varphi}\eta=f(\lambda)\varphi^\prime(\lambda)W_{g,\psi}\eta.
\end{equation*}
This shows that $W_{g,\psi}\eta$ is an eigenvector corresponding to the eigenvalue $f(\lambda)\varphi^\prime(\lambda)$. Therefore, $W_{g,\psi}\eta=\gamma_2\eta$. Thus we have 
\begin{equation}
\begin{split}
W_{g,\psi}\frac{1}{1-\lambda z}\left(\frac{\lambda-z}{1-\lambda z} \right)&=\gamma_2\frac{1}{1-\lambda z}\left(\frac{\lambda-z}{1-\lambda z} \right)\\
g(z)\frac{1}{1-\lambda \psi(z)}\left(\frac{\lambda-\psi(z)}{1-\lambda \psi(z)} \right)&=\gamma_2\frac{1}{1-\lambda z}\left(\frac{\lambda-z}{1-\lambda z} \right)\\
W_{g,\psi}K_{\bar{\lambda}}(z)\left(\frac{\lambda-\psi(z)}{1-\lambda \psi(z)} \right)&=\gamma_2\frac{1}{1-\lambda z}\left(\frac{\lambda-z}{1-\lambda z} \right)\\
\gamma_1 K_{\bar{\lambda}}(z)\left(\frac{\lambda-\psi(z)}{1-\lambda \psi(z)} \right)&=\gamma_2\frac{1}{1-\lambda z}\left(\frac{\lambda-z}{1-\lambda z} \right)\\
\gamma_1 K_{\bar{\lambda}}(z)\left(\frac{\lambda-\psi(z)}{1-\lambda \psi(z)} \right)&=\gamma_2K_{\bar{\lambda}}(z)\left(\frac{\lambda-z}{1-\lambda z} \right)\\
 \left(\frac{\lambda-\psi(z)}{1-\lambda \psi(z)} \right)&=\alpha \left(\frac{\lambda-z}{1-\lambda z} \right).
\end{split}
\end{equation}	
Thus we conclude that $\psi(z)=\frac{(\lambda^2-\alpha)z+\lambda (\alpha-1)}{\lambda (1-\alpha)z+(\lambda^2\alpha-1)}$.

 Since $W_{g,\psi}K_{\bar{\lambda}}=\gamma_1 K_{\bar{\lambda}}$ for some nonzero $\gamma_1\in \mathbb{C}$, for all $z\in \D$, we have 
 
 \begin{equation*}
 \frac{g(z)}{1-\lambda \psi(z)}=\frac{\gamma_1}{1-\lambda z}.
 \end{equation*}
 If we put $z=\lambda$, we have
 \begin{equation*}
 \frac{g(\lambda)}{1-\lambda \psi(\lambda)}=\frac{\gamma_1}{1-\lambda^2 }.
 \end{equation*}
 Thus by Lemma \ref{common fix point}, we have $\gamma_1=g(\lambda)$. Therefore,
 \begin{equation*}
 \begin{split}
 g(z)&=g(\lambda)\frac{1-\lambda \psi(z)}{1-\lambda z}\\
 &=\frac{g(\lambda)}{1-\lambda z}(1-\lambda \psi(z))\\
 &=\frac{g(\lambda)}{1-\lambda z}\left(1-\lambda \frac{(\lambda^2-\alpha)z+\lambda (\alpha-1)}{\lambda (1-\alpha)z+(\lambda^2\alpha-1)}\right)\\
 &=g(\lambda)\frac{\lambda^2-1}{\lambda (1-\alpha)z+(\lambda^2\alpha-1)}
 \end{split}
 \end{equation*}
 Now set $d_0=\frac{\lambda (\alpha-1)}{\lambda^2\alpha-1}=d_1$, $d_2=\alpha\frac{(\lambda^2-1)^2}{(\lambda^2\alpha-1)^2}$, and $d_3=\frac{\lambda^2-1}{\lambda^2\alpha-1}$, then 
 
 \begin{equation}
     \psi(z)=d_0+\frac{d_2z}{1-d_1z}\quad\text{and}\quad g(z)=g(\lambda)\frac{d_3}{1-d_1z}.
 \end{equation}

  On the other hand, consider the converse assertions to be true. Since $W_{g, \psi} W_{f, \varphi}=$ $M_{g(f\circ \psi)} C_{\varphi \circ \psi}$ and $W_{f, \varphi} W_{g, \psi}=M_{f(g \circ \phi)} C_{\psi \circ \phi}$, it suffices to show that $g(f \circ \psi)=f(g \circ \varphi)$ and $\varphi \circ \psi=\psi \circ \varphi$. It is clear when $\lambda=0$. Assume $\lambda \neq 0$.
 
 \begin{equation}
 \begin{split}
 (g(f \circ \psi))(z) &=\frac{g(\lambda) d_{3}}{1-d_{1} z} \times\frac{f(0)}{1-a_{0} \psi(z)} \\
 &=\frac{g(\lambda) d_{3}}{1-d_{1} z}\times \frac{f(0)}{1-a_{0} \frac{d_{0}+\left(d_{2}-d_{0} d_{1}\right) z}{1-d_{1} z}} \\
 &=\frac{g(\lambda) d_{3} f(0)}{\left(1-a_{0} d_{0}\right)-\left[d_{1}+a_{0}\left(d_{2}-d_{0} d_{1}\right)\right] z} .
 \end{split}
 \end{equation}
 
  \begin{equation}
  \begin{split}
  (f(g \circ \varphi))(z) &=\frac{f(0)}{1-a_{0} z} \times\frac{g(\lambda) d_{3}}{1-d_{1} \varphi(z)} \\ 
   &=\frac{g(\lambda) d_{3} f(0)}{\left(1-a_0d_{1}\right)-\left[a_{0}+d_{1}\left(a_{1}-a_{0}^{2}\right)\right] z}
  \end{split}
   \end{equation}
 Therefore, by using (\ref{relation1}) of Lemma \ref{relation among the coefficients}, we conclude that $(g(f \circ \psi))(z)=(f(g \circ \varphi))(z)$.
   
  \begin{equation}
  \begin{split}
 (\varphi \circ \psi)(z) &=\frac{a_{0}+(a_{1}- a_0^2) \psi(z)}{1-a_0\psi(z)} \\
 &=\frac{a_{0}+\left(a_{1}-a_0^{2}\right) d_{0}+\left[-a_{0} d_{1}+\left(a_{1}-a_0^{2}\right)\left(d_{2}-d_{0} d_{1}\right)\right] z}{1-a_0 d_{0}+\left[-d_{1}-a_0\left(d_{2}-d_0 d_{1}\right)\right] z} .
  \end{split}
  \end{equation}

  and
  \begin{equation}
 \begin{split}
  (\psi \circ \varphi)(z) &=\frac{d_{0}+\left(d_{2}-d_{0} d_{1}\right) \varphi(z)}{1-d_{1} \varphi(z)}\\
  &=\frac{d_{0}+\left(d_{2}-d_{0} d_{1}\right) \frac{a_{0}+\left(a_{1}-a_0^{2}\right) z}{1-a_{0} z}}{1-d_{1} \frac{a_{0}+\left(a_{1}-a_0^{2}\right) z}{1-a_0 z}} \\
  &=\frac{d_{0}+\left(d_{2}-d_{0} d_{1}\right) a_{0}+\left[- a_0d_{0}+\left(d_{2}-d_{0} d_{1}\right)\left(a_{1}-a_0^{2}\right)\right] z}{1-a_{0}d_{1} +\left[-a_0-d_{1}\left(a_{1}-a_0^{2}\right)\right] z} .
 \end{split}
  \end{equation} 
   
   Again using (\ref{relation1}) of Lemma \ref{relation among the coefficients}, we obtain $(\varphi \circ \psi)(z)=(\psi \circ \varphi)(z)$. This completes the proof.

\end{proof}

Using \cite[Remark 4.8]{TLe_severalvariables} and Remark \ref{lambda real}, we can easily prove the following:

\begin{corollary}
	Let $g\in H^\infty$ and $\psi$ be an analytic map of $\D$ into itself. Assume that $W_{g,\psi}\in \{W_{f,\varphi}\}^\prime$, where $W_{f,\varphi}$ is complex symmetric with the conjugation $\mathcal{J}$. If $\varphi$ is not an elliptic automorphism  and has a fixed point, say $\lambda$ in $\D$. Then the following holds true:
	\begin{enumerate}
		\item\label{normal} If $\lambda$ is real, then $W_{g,\psi}$ is normal.
		\item\label{self-adjoint} If both $\lambda$ and $\alpha$ are real, then $W_{g,\psi}$ is self-adjoint.
		\item\label{symmetric} If $|d_0|<1$ and $2|d_0+\bar{d}_0\left(d_2-d^2_0\right)|\leq 1-|d_2-d_0^2|$, then $W_{g,\psi}$ is $\mathcal{J}$-symmetric.
	\end{enumerate}
\end{corollary}

\begin{corollary}
    Let $g\in H^\infty$ and $\psi$ be an analytic map of $\D$ into itself. Assume that $W_{g,\psi}\in \{W_{f,\varphi}\}^\prime$ where $W_{f,\varphi}$ is complex symmetric with the conjugation $\mathcal{J}$. If $\varphi$ has a fixed point, say $\lambda$ in $\D$, then the space $\clm=\{ h\in H^2:h(\lambda)=0\} $ is invariant subspace for $W_{g,\psi}$.
\end{corollary}
\begin{proof}
    By Lemma \ref{common fix point}, $\psi$ has the same fixed point $\lambda$. Therefore, for any $h\in \clm$
    \begin{equation*}
    W_{g,\psi}h(\lambda)=g(\lambda)h(\psi(\lambda))=g(\lambda)h(\lambda)=0
    \end{equation*}
\end{proof}
\begin{remark}
 \begin{enumerate}
     \item $\clm=B_\lambda H^2$, where $B_\lambda(z)=\frac{z-\lambda}{1-\bar{\lambda}z}$ and $B_\lambda H^2$ is invariant for $C_\psi$. Therefore, by employing \cite[Theorem 2.3]{Bose:Muthukumar:Sarkar}, we can deduce that $\frac{B_\lambda\circ\psi}{B_\lambda}$ belongs to the class denoted as $\cls(\D)$. In this context, $\cls(\D)$ refers to the collection of functions that are both holomorphic and have a modulus bounded by one on the domain $\D$, formally defined as:
\begin{equation*}
\cls(\D)=\{g\in H^\infty(\D):\|g\|_\infty:=\sup_{z\in\D}|g(z)|\leq 1 \}.
\end{equation*}
This particular set is recognized as the Schur class, and its elements are known as Schur functions.
\item If we assume that the weight function $\psi$ is identically $1$, and the composition operator $C_\vp$ is complex symmetric, then the symbol $\vp$ possess a fix point. Therefore, we can relax the assumption that $\vp$ has a fixed point from our results. 
 \end{enumerate}   
\end{remark}

\begin{center}
	\textbf{Acknowledgements}
\end{center}
The research of the author is supported by the NBHM postdoctoral fellowship, Department of Atomic Energy (DAE), Government of India (File No: 0204/16(21)/2022/R\&D-II/11995) and the INSPIRE grant of Dr. Srijan Sarkar (Ref: DST/INSPIRE/04/2019/000769), Department of Science \& Technology (DST), Government of India.

\bibliographystyle{plain}

\end{document}